\documentclass[12pt,a4paper]{amsart}
\usepackage[english]{babel}
\usepackage{amsmath,amssymb,amsthm,latexsym,MnSymbol}
\usepackage[latin1]{inputenc}
\usepackage{url}
\usepackage[pdftex]{graphicx}
\usepackage{tikz}
\usepackage[pdftex]{hyperref}
\usepackage[all]{xy}

\newcommand{\ZZ}{\mathbf{Z}}
\newcommand{\Zhat}{\widehat{\ZZ}}
\newcommand{\QQ}{\mathbf{Q}}
\newcommand{\Qb}{\overline{\QQ}}

\newcommand{\RR}{\mathbf{R}}
\newcommand{\CC}{\mathbf{C}}

\newcommand{\A}{\mathbf{A}}

\newcommand{\TT}{\mathbf{T}}
\newcommand{\T}{\mathcal{T}}
\newcommand{\TG}{\T G}

\newcommand{\Dvec}{\vec{D_E}}

\DeclareMathOperator{\dv}{div}
\DeclareMathOperator{\darg}{darg}

\DeclareMathOperator{\End}{End}
\DeclareMathOperator{\Gal}{Gal}
\DeclareMathOperator{\GL}{GL}
\DeclareMathOperator{\Hom}{Hom}
\DeclareMathOperator{\id}{id}
\DeclareMathOperator{\Ind}{Ind}

\DeclareMathOperator{\pr}{pr}

\DeclareMathOperator{\reg}{reg}
\DeclareMathOperator{\Res}{Res}

\newtheorem{thm}{Theorem}

\newtheorem{lem}[thm]{Lemma}
\newtheorem{pro}[thm]{Proposition}

\newtheorem*{cor*}{Corollary}
\theoremstyle{definition}
\newtheorem{definition}[thm]{Definition}

\theoremstyle{remark}

\newtheorem*{remarks*}{Remarks}

\title{On Zagier's conjecture for base extensions of elliptic curves}

\author{François Brunault}

\date{\today}

\address{ÉNS Lyon, UMPA, 46 allée d'Italie, 69007 Lyon, France}

\email{francois.brunault@ens-lyon.fr}

\subjclass[2000]{11G40, 11G55, 19F27}

\begin{document}

\begin{abstract}
Let $E$ be an elliptic curve over $\QQ$, and let $F$ be a finite abelian extension of $\QQ$. Using Beilinson's theorem on a suitable modular curve, we prove a weak version of Zagier's conjecture for $L(E/F,2)$, where $E/F$ is the base extension of $E$ to $F$.
\end{abstract}

\maketitle

\section*{Introduction}

Zagier conjectured in \cite{zagier:zc} very deep relations between special values of zeta functions at integers, special values of polylogarithms at algebraic arguments and $K$-theory. While the original conjectures concerned the Dedekind zeta function of a number field (as well as Artin $L$-functions), theoretical and numerical results by many authors suggested an extension of the conjectures to elliptic curves. A precise formulation for elliptic curves over number fields was given by Wildeshaus in \cite{wildeshaus:ezc}. The conjecture on $L(E,2)$, where $E$ is an elliptic curve over $\QQ$, was proved by Goncharov and Levin in \cite{goncharov-levin}. In this article, we prove an analogue of Goncharov and Levin's result for the base extension of $E$ to an arbitrary abelian number field.

Let $E$ be an elliptic curve defined over $\QQ$. Let $F \subset \Qb$ be a finite abelian extension of $\QQ$, of degree $d \geq 1$. Let $G=\Gal(F/\QQ)$ be the Galois group of $F$, and let $\widehat{G}$ be its group of $\Qb^\times$-valued characters. The $L$-function $L(E/F,s)$ of the base change of $E$ to $F$ admits a factorization $\prod_{\chi \in \widehat{G}} L(E \otimes \chi,s)$, where each factor has an analytic continuation to $\CC$ with a simple zero at $s=0$. The functional equation relates $L(E/F,2)$ with the leading term of $L(E/F,s)$ at $s=0$.

Fix an isomorphism $E(\CC) \cong \CC/(\ZZ+\tau \ZZ)$ ($\tau \in \CC$, $\Im(\tau)>0$) which is compatible with complex conjugation. Let $D_E$ (resp. $J_E$) be the Bloch elliptic dilogarithm (resp. its ``imaginary'' cousin) on $E(\CC)$ (see \S \ref{section reg map}-\ref{section regEF} for the definitions). Fix an embedding $\iota : \Qb \hookrightarrow \CC$, so that $E(\Qb)$ embeds naturally in $E(\CC)$. Note that $D_E$ and $J_E$ induce linear maps on $\ZZ[E(\Qb)]$. Let $\ZZ[E(\Qb)]^{G_F}$ be the group of divisors on $E(\Qb)$ which are invariant under $G_F:=\Gal(\Qb/F)$. It carries a natural action of $G$. The main theorem of this article can be stated as follows.

\begin{thm}\label{LEchi0 thm}
There exists a divisor $\ell \in \ZZ[E(\Qb)]^{G_F}$ such that for every $\chi \in \widehat{G}$, we have
\begin{equation}\label{LEchi0 formula}
L'(E \otimes \chi,0) \sim_{\QQ^\times} \begin{cases} \frac{1}{\pi} \sum_{\sigma \in G} \chi(\sigma) D_E(\ell^\sigma) & \textrm{if } \chi \textrm{ is even},\\
\frac{1}{\pi \Im(\tau)} \sum_{\sigma \in G} \chi(\sigma) J_E(\ell^\sigma) & \textrm{if } \chi \textrm{ is odd.}
\end{cases}
\end{equation}
\end{thm}

Using the Dedekind-Frobenius formula for group determinants, we deduce from Theorem \ref{LEchi0 thm} the following result. Write $G=\{\sigma_1,\ldots,\sigma_d\}$ if $F$ is real, and $G=\{\sigma_1,\overline{\sigma_1},\ldots,\sigma_{d/2},\overline{\sigma_{d/2}}\}$ if $F$ is complex.

\begin{cor*}[Weak version of Zagier's conjecture for $L(E/F,2)$]\label{LEF2 thm}
\quad

Let $\ell \in \ZZ[E(\Qb)]^{G_F}$ be a divisor satisfying the identities (\ref{LEchi0 formula}) of Theorem 1. Put $\ell_i = \ell^{\sigma_i^{-1}}$. If $F$ is real, we have
\begin{equation}\label{LEF2 formula real}
L(E/F,2) \sim_{\QQ^\times} \pi^d \cdot \det \bigl(D_E(\ell_i^{\sigma_j})\bigr)_{1 \leq i,j \leq d}.
\end{equation}

If $F$ is complex, we have
\begin{equation}\label{LEF2 formula complex}
L(E/F,2) \sim_{\QQ^\times} \frac{\pi^d}{\Im(\tau)^{d/2}} \cdot \det \bigl(D_E(\ell_i^{\sigma_j})\bigr)_{1 \leq i,j \leq d/2} \cdot \det \bigl(J_E(\ell_i^{\sigma_j})\bigr)_{1 \leq i,j \leq d/2}.
\end{equation}
\end{cor*}

\begin{remarks*}
\begin{enumerate}
\item Wildeshaus's formulation of the conjecture \cite[Conjecture, Part 2, p. 366]{wildeshaus:ezc} uses Kronecker doubles series instead of $D_E$ and $J_E$. The link between these objects is classical (see the proof of Prop \ref{regulateur Etau}). We have chosen here to formulate our results in terms of $D_E$ and $J_E$ because these functions are easier to compute numerically and make apparent the distinction according to the parity of $\chi$.

\item Because of the definition of $\ell_i$, the determinant appearing in (\ref{LEF2 formula real}) is a group determinant, indexed by $G$. In fact, the eigenvalues of the matrix $\bigl(D_E(\ell_i^{\sigma_j})\bigr)$ are precisely the sums $\sum_{\sigma \in G} \chi(\sigma) D_E(\ell^\sigma)$ appearing in Theorem \ref{LEchi0 thm}. This reflects the factorization of the $L$-function of $E_F$ as the product of the twisted $L$-functions of $E$.

\item The work of Goncharov and Levin \cite{goncharov-levin} implies that the divisor $\ell$ produced by Theorem \ref{LEchi0 thm} satisfies the conditions \cite[(2)-(4)]{goncharov-levin}. Following \cite{zagier-gangl}, let $\mathcal{A}_{E/F} \subset \ZZ[E(\Qb)]^{G_F}$ be the subgroup of divisors satisfying these conditions. The strong version of Zagier's conjecture predicts that if $F$ is real (resp. complex), then \emph{for any} divisors $\ell_1,\ldots,\ell_d \in \mathcal{A}_{E/F}$ (resp. $\ell_1,\ldots,\ell_{d/2} \in \mathcal{A}_{E/F}$), the right-hand side of (\ref{LEF2 formula real}) (resp. (\ref{LEF2 formula complex})) is a rational multiple of $L(E/F,2)$ (maybe equal to zero). Unfortunately, and as in the case where the base field is $\mathbf{Q}$, this strong conjecture is beyond the reach of current technology.
\end{enumerate}
\end{remarks*}

In order to prove Theorem \ref{LEchi0 thm}, we prove a weak version of Beilinson's conjecture for the special value $L^{(d)}(E/F,0)$ (see \S \ref{section regEF} for the definition of the objects involved in the following theorem).

\begin{thm}\label{regEF thm}
There exists a subspace $\mathcal{P}_{E/F} \subset H^2_{\mathcal{M}/\ZZ}(E_F,\QQ(2))$ such that $R_{E/F} := \reg_{E/F} (\mathcal{P}_{E/F})$ is a $\QQ$-structure of $H^1(E_F(\CC),\RR)^-$ and
\begin{equation}\label{det REF}
\det (R_{E/F}) = L^{(d)}(E/F,0) \cdot \det (H^1(E_F(\CC),\QQ)^-).
\end{equation}
\end{thm}

We prove Theorem \ref{regEF thm} by using Beilinson's theorem on a suitable modular curve. More precisely, we make use of a result of Schappacher and Scholl \cite{schappacher-scholl} on the (non geometrically connected) modular curve $X_1(N)_F$, where $N$ is the conductor of $E$. We therefore need to work in the adelic setting. We establish a divisibility statement in the Hecke algebra of $X_1(N)_F$ in order to get the desired result for $E_F$.

The methods used in this article are of inexplicit nature and do not give rise, in general, to explicit divisors. However, Theorem \ref{LEchi0 thm} and its corollary can be made explicit in the particular case of the elliptic curve $E=X_1(11)$ and the abelian extension $F=\QQ(\zeta_{11})^+$. In this case we may choose $\ell$ to be supported in the cuspidal subgroup of $E$. The tools for proving this are Kato's explicit version of Beilinson's theorem for the modular curve $X_1(N)_{\QQ(\zeta_m)}$, the work of the author \cite{brunault:smf}, as well as a technique used by Mellit \cite{mellit} to get new relations between values of the elliptic dilogarithm. We hope to give soon an expanded account of this example.

The organization of the article is as follows. In \S 1, we recall well-known facts about $L(E/F,s)$. In \S 2 and \S 3, we recall the definition of the regulator map and we compute it for $E_F$. In \S 4, we explain the adelic setting for modular curves. In \S 5, we prove the divisibility we need in the Hecke algebra. Finally, we give in \S 6 the proofs of the main results. We conclude with some open questions and remarks.

\textbf{Acknowledgements.}
I would like to thank Anton Mellit for the very inspiring discussions which led to the discovery of the example alluded above, which in turn motivated all the results presented here.

\section{The $L$-function of $E_F$}\label{section LEF}

By the Kronecker-Weber theorem, we have $F \subset \QQ(\zeta_m)$ for some $m \geq 1$, so that $G$ is a quotient of $(\ZZ/m\ZZ)^\times$ and $\widehat{G}$ can be identified with a subgroup of the Dirichlet characters modulo $m$.

Let $f = \sum_{n \geq 1} a_n q^n \in S_2(\Gamma_0(N))$ be the newform associated to $E$. For any $\chi \in \widehat{G}$, define $L(E \otimes \chi ,s) := L(f \otimes \chi,s)$, where $f \otimes \chi$ is the unique newform of weight 2 whose $p$-th Fourier coefficient is $a_p \chi(p)$ for every prime $p \nupdownline Nm$. The $L$-function of $E_F$ has the following description.

\begin{pro}\label{pro LEF}
The following identity holds :
\begin{equation}\label{LEF}
L(E/F,s) = \prod_{\chi \in \widehat{G}} L(f \otimes \chi,s).
\end{equation}
\end{pro}

\begin{proof}
Let $\rho = (\rho_\ell)_{\ell}$ be the compatible system of $2$-dimensional $\ell$-adic representations of $G_\QQ$ attached to $f$ by Deligne \cite{deligne:GL2}. By modularity $L(E/F,s) = L(\rho |_{G_F},s)$. Using Artin's formalism for $L$-functions, we have
\begin{equation}
L(\rho |_{G_F},s) = L(\Ind_{G_F}^{G_\QQ} (\rho |_{G_F}),s).
\end{equation}
If $\mathbf{1}_{G_F}$ denotes the trivial representation of $G_F$, we have
\begin{align}
\nonumber \Ind_{G_F}^{G_\QQ} (\rho |_{G_F}) & = \Ind_{G_F}^{G_\QQ} (\mathbf{1}_{G_F} \otimes \Res_{G_\QQ}^{G_F} \rho )\\
 & \cong \Ind_{G_F}^{G_\QQ} (\mathbf{1}_{G_F}) \otimes  \rho\\
\nonumber & \cong \bigoplus_{\chi \in \widehat{G}} \rho \otimes \chi.\label{eq6}
\end{align}
(Here we chose embeddings $\Qb \hookrightarrow \overline{\QQ_\ell}$.) Finally, since an irreducible $\ell$-adic representation of $G_\QQ$ is determined by the traces of all but finitely many Frobenius elements, the compatible system associated to $f \otimes \chi$ is $\rho \otimes \chi$, so that $L(\rho \otimes \chi,s) = L(f \otimes \chi,s)$ for any $\chi \in \widehat{G}$.
\end{proof}

Since each $L(f \otimes \chi,s)$ has a simple zero at $s=0$, we get
\begin{equation}\label{LEF0}
\frac{L^{(d)}(E/F,0)}{d!} = \prod_{\chi \in \widehat{G}} L'(f \otimes \chi,0),
\end{equation}
where $L^{(d)}(E/F,0)$ denotes the $d$-th derivative at $0$.

\begin{pro}\label{LEF20}
We have $L(E/F,2) \sim_{\QQ^\times} \pi^{2d} L^{(d)}(E/F,0)$.
\end{pro}

\begin{proof}
Let $N_{f \otimes \chi}$ be the level of the newform $f \otimes \chi$. Putting $\Lambda(f \otimes \chi,s) = N_{f \otimes \chi}^{s/2} (2\pi)^{-s} \Gamma(s) L(f \otimes \chi,s)$, we have \cite[\S 5]{diamond-im}
\begin{equation}\label{fun eq fchi}
\Lambda(f \otimes \chi,s) = -w_{f \otimes \chi} \Lambda(f \otimes \overline{\chi},2-s) \qquad (s \in \CC)
\end{equation}
where $w_{f \otimes \chi}$ is the pseudo-eigenvalue of $f \otimes \chi$ with respect to the Atkin-Lehner involution of level $N_{f \otimes \chi}$. Note that (\ref{fun eq fchi}) implies $w_{f \otimes \chi} w_{f \otimes \overline{\chi}}=1$. Letting $w=\prod_{\chi \in \widehat{G}} w_{f \otimes \chi}$, we have
\begin{equation}
w^2=\prod_{\chi \in \widehat{G}} w_{f \otimes \chi} w_{f \otimes \overline{\chi}} = 1
\end{equation}
so that $w=\pm 1$. Moreover $\Lambda(f \otimes \chi,0) = L'(f \otimes \chi,0)$ and $\Lambda(f \otimes \overline{\chi},2)=(N_{f \otimes \overline{\chi}}/4\pi^2) L(f \otimes \overline{\chi},2)$. Taking the product over $\chi \in \widehat{G}$ yields the result.
\end{proof}

\section{The regulator map on Riemann surfaces} \label{section reg map}

In this section, we recall the definition of the regulator map on compact Riemann surfaces \cite[\S 1]{esnault-viehweg}, and its computation in the case of elliptic curves.

Let $X$ be a compact connected Riemann surface, and $\mathcal{M}(X)$ be its field of meromorphic functions. For any $f,g \in \mathcal{M}(X)^\times$, consider the $1$-form
\begin{equation}
\eta(f,g) := \log |f| \cdot \darg(g) - \log |g| \cdot \darg(f).
\end{equation}
For any $f \in \mathcal{M}(X) \backslash \{0,1\}$, the differential form $\eta(f,1-f)$ is exact on $X \backslash f^{-1}(\{0,1,\infty\})$. More precisely $\eta(f,1-f) = d(D \circ f)$, where $D$ is the Bloch-Wigner dilogarithm function \cite{zagier:blochwigner}. Let $K_2(\mathcal{M}(X))$ be the Milnor $K_2$-group associated to $\mathcal{M}(X)$. The \emph{regulator map} on $X$ is the unique linear map
\begin{equation}\label{regXC}
\reg_X  : K_2(\mathcal{M}(X)) \to H^1(X,\RR)
\end{equation}
such that for any $f,g \in \mathcal{M}(X)^\times$ and any holomorphic $1$-form $\omega$ on $X$, we have
\begin{equation}\label{def regXC}
\int_X \reg_X \{f,g\} \wedge \omega = \frac{1}{2\pi} \int_X \eta(f,g) \wedge \omega.
\end{equation}
The map $\reg_X$ is well-defined by exactness of $\eta(f,1-f)$ and Stokes' theorem. The construction of $\reg_X$ easily extends to the case where $X$ is compact but not connected. Indeed, put $\mathcal{M}(X) := \prod_{i=1}^r \mathcal{M}(X_i)$ where $X_1,\ldots,X_r$ are the connected components of $X$. Then $K_2(\mathcal{M}(X)) \cong \bigoplus_i K_2(\mathcal{M}(X_i))$ as well as $H^1(X,\RR) \cong \bigoplus_i H^1(X_i,\RR)$, and we define $\reg_X$ to be the direct sum of the maps $\reg_{X_i}$ for $1 \leq i \leq r$.

Let us recall the classical computation of the regulator map on a complex torus \cite[\S 4]{beilinson:1}. Let $E_\tau := \CC/(\ZZ+\tau \ZZ)$ with $\tau \in \CC$, $\Im(\tau)>0$. The map $z \mapsto \exp(2i\pi z)$ induces an isomorphism $E_\tau \cong \CC^\times / q^\ZZ$, where $q := \exp(2i\pi \tau)$. Let $D_q : E_\tau \to \RR$ be the Bloch elliptic dilogarithm, defined by $D_q([x])= \sum_{n=-\infty}^{\infty} D(xq^n)$ for any $x \in \CC^\times$. We will also use the function $J_q : E_\tau \to \RR$, which is defined as follows. Let $J : \CC^\times \to \RR$ be the function defined by $J(x)=\log |x| \cdot \log |1-x|$ if $x \neq 1$, and $J(1)=0$. Following \cite{zagier:blochwigner}, we put
\begin{equation}
J_q([x]) = \sum_{n=0}^{\infty} J(xq^n) - \sum_{n=1}^{\infty} J(x^{-1} q^n) + \frac13 \log^2 |q| \cdot B_3\bigl(\frac{\log |x|}{\log |q|}\bigr) \qquad (x \in \CC^\times)
\end{equation}
where $B_3 =X^3-\frac32 X^2 + \frac{X}{2}$ is the third Bernoulli polynomial. The function $J_q$ is well-defined since $J(x)+J(\frac{1}{x})=\log^2 |x|$ and $B_3(X+1)-B_3(X)=3X^2$. Both functions $D_q$ and $J_q$ extend to linear maps $\ZZ[E_\tau] \to \RR$, by setting $D_q(\sum_{i} n_i [P_i]) := \sum_i n_i D_q(P_i)$ and similarly for $J_q$.

\begin{definition} For any $f,g \in \mathcal{M}(E_\tau)^\times$ with divisors $\dv(f) = \sum_i m_i [P_i]$ and $\dv(g) = \sum_j n_j [Q_j]$, the divisor $\beta(f,g) \in \ZZ[E_\tau]$ is given by
\begin{equation}
\beta(f,g) = \sum_{i,j} m_i n_j [P_i-Q_j].
\end{equation}
\end{definition}
The following classical result expresses the regulator map on $E_\tau$ in terms of $D_q$ and $J_q$.

\begin{pro}\label{regulateur Etau}
For any $f,g \in \mathcal{M}(E_\tau)^\times$, we have
\begin{equation}
\int_{E_\tau} \eta(f,g) \wedge dz = (D_q-iJ_q)(\beta(f,g)).
\end{equation}
\end{pro}

\begin{proof}
We have $\int_{E_\tau} \eta(f,g) \wedge dz = -\frac{\Im(\tau)^2}{\pi} K_{2,1,\tau}(\beta(f,g))$ by \cite[\S 4.3]{beilinson:1} and \cite[(6.2)]{deninger:hecke1}, where $K_{2,1,\tau}$ is the linear extension of the following Eisenstein-Kronecker series on $E_\tau$ :
\begin{equation}
K_{2,1,\tau}(z) := \sum_{\substack{\lambda \in \ZZ+\tau \ZZ\\ \lambda \neq 0}} \frac{\exp \bigl( \frac{2i\pi}{\tau-\overline{\tau}} (z \overline{\lambda} - \overline{z} \lambda)\bigr)}{\lambda^2 \overline{\lambda}} \qquad (z \in \CC/(\ZZ+\tau \ZZ)).
\end{equation}
The result now follows from the formula $-\frac{\Im(\tau)^2}{\pi} K_{2,1,\tau} = D_q - i J_q$, for which we refer to \cite[Thm 10.2.1]{bloch:irvine} and \cite[\S 2, p. 616]{zagier:blochwigner}.
\end{proof}

\section{The regulator map on $E_F$} \label{section regEF}

Let $X$ be a connected (but not necessarily geometrically connected) smooth projective curve over $\QQ$. Its function field $\QQ(X)$ embeds into $\mathcal{M}(X(\CC))$, so we get a natural map $K_2(\QQ(X)) \to K_2(\mathcal{M}(X(\CC)))$. Let $c$ denote the complex conjugation on $X(\CC)$. For any $f,g \in \QQ(X)^\times$, we have $c^* \eta(f,g) = -\eta(f,g)$, so that (\ref{regXC}) induces a map
\begin{equation}\label{reg K2QX}
K_2(\QQ(X)) \to H^1(X(\CC),\RR)^{-},
\end{equation}
where $(\cdot)^-$ denotes the $(-1)$-eigenspace of $c^*$.

Let $K_2(X)$ be the Quillen algebraic $K_2$-group associated to $X$. Recall that the motivic cohomology group $H^2_{\mathcal{M}}(X,\QQ(2)) := K_2^{(2)}(X)$ is defined as the second Adams eigenspace of $K_2(X) \otimes \QQ$. The exact localization sequence in $K$-theory yields a canonical injective map $K_2(X) \otimes \QQ \hookrightarrow K_2(\QQ(X)) \otimes \QQ$ which is compatible with the Adams operations, so that in fact $K_2^{(2)}(X)=K_2(X) \otimes \QQ$. The \emph{integral subspace} $H^2_{\mathcal{M}/\ZZ}(X,\QQ(2)) \subset H^2_{\mathcal{M}}(X,\QQ(2))$ is the image of the map $K_2(\mathcal{X}) \otimes \QQ \to K_2(X) \otimes \QQ$ for any proper regular model $\mathcal{X}/\ZZ$ of $X$ (see \cite{scholl:integral_elements} for a definition in a more general setting). Tensoring (\ref{reg K2QX}) with $\QQ$ and restricting to the integral subspace gives the \emph{Beilinson regulator map} on $X$ :
\begin{equation}
\reg_X : H^2_{\mathcal{M}/\ZZ}(X,\QQ(2)) \to H^1(X(\CC),\RR)^-.
\end{equation}
Note that the real vector space $H^1(X(\CC),\RR)^-$ admits the natural $\QQ$-structure $H_X := H^1(X(\CC),\QQ)^-$.

Any finite morphism $\varphi : X \to Y$ between smooth projective curves over $\QQ$ induces maps $\varphi^* : K_2(Y) \to K_2(X)$ and $\varphi_* : K_2(X) \to K_2(Y)$, the latter being defined by $K_2(X) \xrightarrow{\cong} K'_2(X) \xrightarrow{\varphi_*} K'_2(Y) \xleftarrow{\cong} K_2(Y)$. It is known that $\varphi^* \otimes \QQ$ and $\varphi_* \otimes \QQ$ preserve the integral subspaces \cite[Thm 1.1.6(i)]{scholl:integral_elements}. Moreover, the Beilinson regulator maps associated to $X$ and $Y$ are compatible with $\varphi^*$ and $\varphi_*$ (this can be seen at the level of Riemann surfaces).

Let us return to our elliptic curve $E$. Fix an isomorphism $E(\CC) \cong E_\tau$ which is compatible with complex conjugation, and let $q=\exp(2i\pi \tau)$. Let $D_E$ and $J_E$ be the real-valued functions on $E(\CC)$ induced by $D_q$ and $J_q$ respectively\footnote{The lattice $\ZZ+\tau \ZZ$ is uniquely determined by $E$, and $q$ is a well-defined real number such that $0<|q|<1$. But the pair $(D_E,J_E)$ is defined only up to sign (choosing an isomorphism $E(\CC) \cong E_{\tau}$ amounts to specifying an orientation of $E(\RR)$).}. The space $H^1(E(\CC),\QQ)^{\pm}$ is generated by the $1$-form $\eta^\pm$, with
\begin{equation}\label{eta pm}
\eta^+ =dz+d\overline{z} \qquad \textrm{and} \qquad \eta^- = \frac{dz-d\overline{z}}{\tau-\overline{\tau}}.
\end{equation}

%

\begin{lem}\label{regEC fg}
Let $f,g \in \CC(E)^\times$ and $\ell = \beta(f,g)$. We have
\begin{equation}
\reg_{E(\CC)} \{f,g\} = -\frac{1}{2\pi} \bigl( D_E(\ell) \cdot \eta^- + \frac{J_E(\ell)}{2\Im(\tau)} \cdot \eta^+ \bigr).
\end{equation}
\end{lem}

\begin{proof}
Put $\reg_{E(\CC)} \{f,g\} = a^+ \eta^+ + a^- \eta^-$ with $a^+,a^- \in \RR$. Taking the wedge product with $dz$ and integrating over $E(\CC)$ yields
\begin{equation}
\int_{E(\CC)} \reg_{E(\CC)} \{f,g\} \wedge dz = -a^- + 2i \Im(\tau) a^+.
\end{equation}
Using (\ref{def regXC}) with Prop. \ref{regulateur Etau} and identifying the real and imaginary parts gives the lemma.
\end{proof}

Let $\Sigma$ be the set of embedding of $F$ into $\CC$. The embedding $\iota : \Qb \hookrightarrow \CC$ induces a distinguished element $\iota \in \Sigma$. Note that $E_F(\CC)$ is the disjoint union of $d$ copies of $E(\CC)$, so that
\begin{equation}\label{decomp H1EF}
H^1(E_F(\CC),\RR) \cong \bigoplus_{\psi \in \Sigma} H^1(E(\CC),\RR)
\end{equation}
and $H^1(E_F(\CC),\QQ)$ decomposes accordingly. The group $G$ acts from the right on $E_F = E \times_{\operatorname{Spec} \QQ} \operatorname{Spec} F$. This induces a left action of $G$ on $H^1(E_F(\CC),\QQ)$. For any character $\chi \in \widehat{G}$, consider the idempotent
\begin{equation}
e_\chi := \frac{1}{|G|} \sum_{\sigma \in G} \chi(\sigma) \cdot [\sigma] \in \Qb[G].
\end{equation}
It acts on $H^1(E_F(\CC), \Qb \otimes \RR)$. For any $\psi \in \Sigma$, let $\eta^{\pm}(\psi)$ be the $1$-form $\eta^{\pm}$ sitting in the $\psi$-component of (\ref{decomp H1EF}). Define
\begin{equation}
\eta_\chi = \begin{cases} e_\chi(\eta^-(\iota)) & \textrm{if $\chi$ is even}\\
e_\chi(\eta^+(\iota)) & \textrm{if $\chi$ is odd}.
\end{cases}
\end{equation}

\begin{lem}\label{eta chi}
If $c : E_F(\CC) \to E_F(\CC)$ is the map induced by complex conjugation on $\operatorname{Spec} \CC$, then $c^* \eta_\chi = - \eta_\chi$.
\end{lem}

\begin{proof}
For any $\psi \in \Sigma$, we have $c^* \eta^{\pm}(\psi) = \pm \eta^{\pm}(\overline{\psi})$. It follows that
\begin{align*}
c^* \eta_\chi & = \frac{1}{|G|} \sum_{\sigma \in G} \chi(\sigma) c^* (\sigma \cdot \eta^{-\chi(-1)}(\iota))\\
& = \frac{-\chi(-1)}{|G|} \sum_{\sigma \in G} \chi(\sigma) (\overline{\sigma} \cdot \eta^{-\chi(-1)}(\iota)).
\end{align*}
Since $\chi(-1) \chi(\sigma) = \chi(\overline{\sigma})$, we get the result.
\end{proof}

The map $\beta$ induces a linear map $F(E)^\times \otimes F(E)^\times \to \ZZ[E(\Qb)]^{G_F}$, which we still denote by $\beta$. The following proposition computes explicitly the regulator map associated to $E_F$.

\begin{pro}\label{echi reg}
Let $\gamma \in F(E)^\times \otimes F(E)^\times$ and $\ell = \beta(\gamma)$. For any $\chi \in \widehat{G}$, we have
$e_\chi \reg_{E/F}([\gamma]) = \mu_\chi(\ell) \cdot \eta_\chi$, where $\mu_\chi(\ell) \in \Qb \otimes \RR$ is given by
\begin{equation}\label{eq echi reg}
\mu_\chi(\ell) = \begin{cases}
- \frac{1}{2\pi} \sum_{\sigma \in G} \chi(\sigma) \otimes D_E(\ell^{\sigma}) & \textrm{if } \chi \textrm{ is even},\\
- \frac{1}{4\pi \Im(\tau)} \sum_{\sigma \in G} \chi(\sigma) \otimes J_E(\ell^{\sigma}) & \textrm{if } \chi \textrm{ is odd}.
\end{cases}
\end{equation}
\end{pro}

\begin{proof}
Put $r=\reg_{E/F}([\gamma])$. By Lemma \ref{regEC fg}, the $\psi$-component of $r$ is
\begin{equation}
r_\psi  = -\frac{1}{2\pi} \Bigl( D_E(\psi(\ell)) \cdot \eta^-(\psi) + \frac{J_E(\psi(\ell))}{2\Im(\tau)} \cdot \eta^+(\psi) \Bigr).
\end{equation}
Since $e_\chi(r)$ and $\eta_\chi$ belong to the same $G$-eigenspace, it suffices to compare their $\iota$-components. By definition, we have $(\eta_\chi)_\iota = \frac{1}{|G|} \eta^{\chi(-1)}$. Moreover
\begin{align}\label{echi r}
e_\chi(r)_\iota & = \frac{1}{|G|} \sum_{\sigma \in G} \chi(\sigma) \otimes (\sigma \cdot r)_\iota = \frac{1}{|G|} \sum_{\sigma \in G} \chi(\sigma) \otimes r_{\iota \circ \sigma}\\
& = -\frac{1}{2\pi |G|} \sum_{\sigma \in G} \chi(\sigma) \otimes \bigl( D_E(\ell^\sigma) \cdot \eta^- + \frac{J_E(\ell^\sigma)}{2\Im(\tau)} \cdot \eta^+ \bigr).
\end{align}
But $D_E(\overline{P})=D_E(P)$ and $J_E(\overline{P})=-J_E(P)$ for any $P \in E(\CC)$, so that the terms involving $J_E$ (resp. $D_E$) cancel out if $\chi$ is even (resp. odd).
\end{proof}

\section{Modular curves in the adelic setting}\label{section adelic}

Let $\A_f$ be the ring of finite adèles of $\QQ$. For any compact open subgroup $K \subset \GL_2(\A_f)$, there is an associated smooth projective modular curve $\overline{M}_K$ over $\QQ$. For example $X(N) = \overline{M}_{K(N)}$ and $X_1(N) =\overline{M}_{K_1(N)}$, where
\begin{align}
K(N) & = \ker (\GL_2(\Zhat) \to \GL_2(\ZZ/N\ZZ))\\
K_1(N) & = \Bigl\{g \in \GL_2(\Zhat); g \equiv \begin{pmatrix} * & * \\ 0 & 1 \end{pmatrix} \pmod{N}\Bigr\}.
\end{align}

The Riemann surface $\overline{M}_K(\CC)$ can be identified with the compactification of $\GL_2(\QQ) \backslash (\mathfrak{h}^{\pm} \times \GL_2(\A_f)) / K$. The set of connected components of $\overline{M}_K(\CC)$ is in bijection with $\Zhat^\times / \det(K)$. For any $g \in \GL_2(\A_f)$, we have an isomorphism $g : \overline{M}_K \xrightarrow{\cong} \overline{M}_{g^{-1}Kg}$ over $\QQ$, which is given on the complex points by $(\tau,h) \mapsto (\tau,hg)$. For any compact open subgroups $K' \subset K$ of $\GL_2(\A_f)$, we have a finite morphism $\pi_{K',K} : \overline{M}_{K'} \to \overline{M}_K$.

The Hecke algebra $\mathcal{H}_K$ is the space of functions $K \backslash \GL_2(\A_f)/K \to \Qb$ with finite support, equipped with the convolution product \cite{cartier:corvallis}. It acts on $H^1(\overline{M}_K(\CC),\Qb)$ and $\Omega^1(\overline{M}_K) \otimes \Qb$. Let $\TT_K = \TT_{\overline{M}_K}$ be the image of $\mathcal{H}_K$ in $\End_{\Qb} (\Omega^1(\overline{M}_K) \otimes \Qb)$. Let
\begin{equation} \label{poincare duality}
\langle \cdot,\cdot \rangle : H^1(\overline{M}_K(\CC),\RR)^- \times \bigl(\Omega^1(\overline{M}_K) \otimes \RR\bigr) \to \RR
\end{equation}
be the perfect pairing induced by Poincaré duality. For any $T \in \mathcal{H}_K$, we have $\langle T \eta, \omega \rangle = \langle \eta, T' \omega \rangle$, where $T' \in \mathcal{H}_K$ is defined by $T'(g)=T(g^{-1})$, so that the action of $\mathcal{H}_K$ on $H^1(\overline{M}_K(\CC),\Qb \otimes \RR)^-$ factors through $\TT_K$.

Following \cite[1.1.1]{schappacher-scholl}, let $\mathcal{Q}_{K} \subset K_2(\overline{M}_{K}) \otimes \QQ$ be the subspace of Beilinson elements, and let
\begin{equation}
\mathcal{P}_K = \bigcup_{K' \subset K} (\pi_{K',K})_* \mathcal{Q}_{K'} \subset K_2(\overline{M}_K) \otimes \QQ.\end{equation}
Schappacher and Scholl \cite[1.1.2]{schappacher-scholl} proved that $\mathcal{P}_K \subset H^2_{\mathcal{M}/\ZZ}(\overline{M}_K,\QQ(2))$ and that $\reg_{\overline{M}_K}(\mathcal{P}_K)$ is a $\QQ$-structure of $H^1(\overline{M}_K(\CC),\RR)^-$ whose determinant with respect to the natural $\QQ$-structure $H_{\overline{M}_K}$ is given by the leading term of $L(h^1(\overline{M}_K),s)$ at $s=0$.

In the following, we assume $K = \prod_p K_p$, where $K_p$ a compact open subgroup of $\GL_2(\QQ_p)$. The Hecke algebra then decomposes as a restricted tensor product $\mathcal{H}_K= \bigotimes_p ' \mathcal{H}_{K_p}$. For any prime $p$, let $\widetilde{T}(p) \in \mathcal{H}_K$ (resp. $\widetilde{T}(p,p) \in \mathcal{H}_K$) be the characteristic function of $K \begin{pmatrix} \varpi_p & 0 \\ 0 & 1 \end{pmatrix} K$ (resp. $K \begin{pmatrix} \varpi_p & 0 \\ 0 & \varpi_p \end{pmatrix}$), where $\varpi_p \in \A_f^\times$ has component $p$ at the place $p$, and $1$ elsewhere. Let $T(p)$ (resp. $T(p,p)$) be the image of $\widetilde{T}(p)$ (resp. $\widetilde{T}(p,p)$) in $\TT_K$. When $K$ needs to be specified, we write $T(p)_K$ or $T(p)_{\overline{M}_K}$.

For any integer $M \geq 1$, we let $\mathcal{H}^{(M)}_K \subset \mathcal{H}_K$ be the subalgebra generated by the $\mathcal{H}_{K_p}$ for $p \nupdownline M$. We use the notation $\TT^{(M)}_K$ for the corresponding subalgebra of $\TT_K$.

\begin{lem}\label{lem1}
If $K(M) \subset K$ then $\TT^{(M)}_K$ is in the center of $\TT_K$.
\end{lem}

\begin{proof}
For any prime $p \nupdownline M$, we have $K_p = \GL_2(\ZZ_p)$ and by Satake the map $\Qb[T,S,S^{-1}] \to \mathcal{H}_{K_p}$ given by $T \mapsto \widetilde{T}(p)$ and $S \mapsto \widetilde{T}(p,p)$ is an isomorphism. In particular $\mathcal{H}_{K_p}$ is contained in the center of $\mathcal{H}_K$, whence the result.
\end{proof}

Let $U_F \subset \Zhat^\times$ denote the preimage of $\Gal(\QQ(\zeta_m)/F) \subset (\ZZ/m\ZZ)^\times$ under the natural map $\Zhat^\times \to (\ZZ/m\ZZ)^\times$ (note that $U_F$ does not depend on $m$). For any compact open subgroup $K \subset \GL_2(\A_f)$ with $\det(K)=\Zhat^\times$, let
\begin{equation}
K_F := \{k \in K; \det(k) \in U_F\}.
\end{equation}
Let $\pr : \A_f^\times \to \Zhat^\times$ be the projection associated to the decomposition $\A_f^\times \cong \QQ_{>0} \times \Zhat^\times$.

\begin{definition}
Let $\gamma : \GL_2(\A_f) \to G$ be the composite morphism
\begin{equation}
\GL_2(\A_f) \xrightarrow{\det} \A_f^\times \xrightarrow{\pr} \Zhat^\times \to (\frac{\ZZ}{m\ZZ})^\times \to G.
\end{equation}
\end{definition}
Note that there is an exact sequence
\begin{equation}\label{suite exacte KF}
1 \to K_F \to K \xrightarrow{\gamma |_K} G \to 1.
\end{equation}
The sequence (\ref{suite exacte KF}) induces a right action of $G$ on $\overline{M}_{K_F}$, and thus a left action of $G$ on $\Omega^1(\overline{M}_{K_F})$. Moreover, the curve $\overline{M}_{K_F}$ can be identified with $\overline{M}_K \otimes F$ as a curve over $\QQ$, and we have a bijection
\begin{align}
\label{identif MKF} \overline{M}_{K_F}(\CC) & \xrightarrow{\cong} G \times \overline{M}_K(\CC)\\
\nonumber [\tau,g] & \mapsto (\gamma(g),[\tau,g]).
\end{align}
The action of $G$ on $\overline{M}_{K_F}(\CC)$ corresponds via (\ref{identif MKF}) to the action by translation on the first factor of $G \times \overline{M}_K(\CC)$.

Now let us consider the case $K=K_1(N)$, so that $\overline{M}_{K_F} \cong X_1(N)_F$. By the previous discussion, the image of $G$ in $\End \Omega^1(X_1(N)_F) \otimes \Qb$ is contained in $\TT_{X_1(N) \otimes F}$. In order to ease notations, let $\T = \TT^{(Nm)}_{X_1(N) \otimes F} \subset \End \Omega^1(X_1(N)_F) \otimes \Qb$. Let $\TG$ be the subalgebra of $\TT_{X_1(N) \otimes F}$ generated by $\T$ and $G$.

\begin{lem}\label{lem3}
The algebra $\TG$ is commutative.
\end{lem}

\begin{proof}
Note that $K(Nm) \subset K_1(N)_F$, so $\T$ is commutative and commutes with $G$ by Lemma \ref{lem1}. Since $G$ is abelian, the result follows.
\end{proof}

Since $\Omega^1(X_1(N)_F) \cong \Omega^1(X_1(N)) \otimes F$, we can define the base change morphism $\nu_F : \End \Omega^1(X_1(N)) \to \End \Omega^1(X_1(N)_F)$ by $\nu_F(T)= T \otimes \id_F$. For any $\alpha \in (\ZZ/m\ZZ)^\times$, let $\sigma_\alpha$ be its image in $G$.

\begin{lem}\label{lem4}
For any prime $p \nupdownline Nm$, we have
\begin{align}
\label{nuF Tp}\nu_F \bigl(T(p)_{X_1(N)}\bigr) & = T(p)_{X_1(N) \otimes F} \cdot \sigma_p \in \TG\\
\label{nuF Tpp}\nu_F \bigl(T(p,p)_{X_1(N)}\bigr) & = T(p,p)_{X_1(N) \otimes F} \cdot \sigma_p^2 \in \TG.
\end{align}
\end{lem}

\begin{proof}
Let $g:=\begin{pmatrix} \varpi_p & 0 \\ 0 & 1 \end{pmatrix}$ and $K:=K_1(N) \cap g^{-1}K_1(N)g = K_1(N) \cap K_0(p)$. Note that $\det K=\Zhat^\times$. Consider the following correspondence
\begin{equation}
\xymatrix{
 & \overline{M}_K \ar[ld]_{\alpha} \ar[rd]^{\beta} \\
 X_1(N) \ar@{-->}[rr]^{\widetilde{T}(p)_{X_1(N)}} & & X_1(N)
}
\end{equation}
where $\alpha=\pi_{K,K_1(N)}$ and $\beta = g^{-1} \circ \pi_{K,g^{-1} K_1(N) g} = \pi_{gKg^{-1},K_1(N)} \circ g^{-1}$. Then $T(p)_{X_1(N)} = \beta_* \circ \alpha^*$ on $\Omega^1(X_1(N))$. Similarly $T(p)_{X_1(N) \otimes F}$ is defined by
\begin{equation}
\xymatrix{
 & \overline{M}_{K_F} \ar[ld]_{\alpha_F} \ar[rd]^{\beta_F} \\
 X_1(N)_F \ar@{-->}[rr]^{\widetilde{T}(p)_{X_1(N) \otimes F}} & & X_1(N)_F
}
\end{equation}
where $\alpha_F$ is the natural projection and $\beta_F$ is induced by $g^{-1}$. Using the identification $\overline{M}_{K_F} \cong \overline{M}_K \otimes F$ and the description (\ref{identif MKF}) of the complex points, we obtain $\alpha_F = \alpha \otimes \id_F$ and $\beta_F = \beta \otimes \gamma(g^{-1})$. Since $\gamma(g)=\sigma_{p^{-1}}$, we get $T(p)_{X_1(N) \otimes F} = \nu_F(T(p)_{X_1(N)}) \circ (\sigma_p)_*$ and thus (\ref{nuF Tp}). The proof of (\ref{nuF Tpp}) is similar.
\end{proof}

\section{A divisibility in the Hecke algebra}\label{section phi}

In this section we define and study a projection associated to $E_F$ using the Hecke algebra of $X_1(N)_F$.

Let $\varphi : X_1(N) \to E$ be a modular parametrization of the elliptic curve $E$, and let $\varphi_F : X_1(N)_F \to E_F$ be the base extension of $\varphi$ to $F$. Consider the map $e_F = \frac{1}{\deg \varphi_F} (\varphi_F)^* (\varphi_F)_*$ on $\Omega^1(X_1(N)_F)$.

\begin{lem}\label{lem5}
We have $e_F^2=e_F$ and $e_F \in \TG$.
\end{lem}

\begin{proof}
The first equality follows from $(\varphi_F)_* (\varphi_F)^* = \deg \varphi_F$.

We have $e_F = \nu_F(e)$ where $e= \frac{1}{\deg \varphi} \varphi^* \varphi_* \in \End_\QQ \Omega^1(X_1(N))$. The image of $e$ is the $\QQ$-vector space generated by $\omega_f = 2i\pi f(z)dz$. Since $f$ is a newform of level $N$, the Atkin-Lehner-Li theory implies that $e \in \TT^{(Nm)}_{X_1(N)}$. The result now follows from Lemma \ref{lem4}.
\end{proof}

The space $\Omega = \varinjlim_{K} \Omega^1(\overline{M}_K) \otimes \Qb$ has a natural $\GL_2(\A_f)$-action and decomposes as a direct sum of irreducible admissible representations $\Omega_\pi$ of $\GL_2(\A_f)$. For any $K$ we have $\Omega^K = \Omega^1(\overline{M}_K) \otimes \Qb$. Let $\Pi(K)$ be the set of such $\pi$ satisfying $\Omega_\pi^K \neq \{0\}$. By \cite[p. 393]{langlands}, we have
\begin{equation}\label{dec Omega1}
\Omega^1(\overline{M}_K) \otimes \Qb = \bigoplus_{\pi \in \Pi(K)} \Omega_\pi^K
\end{equation}
where each $\Omega_\pi^K$ is a simple $\TT_K$-module. In particular $\TT_K$ is a semisimple algebra. By Lemma \ref{lem1}, the algebra $\T$ is contained in the center of $\TT_{K_1(N)_F}$. Using \cite[Prop 2.11]{langlands}, we deduce that $\T$ acts by scalar multiplication on each $\Omega_\pi^{K_1(N)_F}$, so there exists a morphism $\theta_{\pi} : \T \to \Qb$ such that $T$ acts as $\theta_\pi(T)$ on $\Omega_\pi^{K_1(N)_F}$. The multiplicity one and strong multiplicity one theorems \cite{piatetski-shapiro} ensure that the characters $(\theta_\pi)_{\pi \in \Pi(K)}$ are pairwise distinct.

For any $\chi \in \widehat{G}$, let $\pi (f \otimes \chi)$ be the automorphic representation of $\GL_2(\A_f)$ corresponding to the modular form $f \otimes \chi$. We have $\pi (f \otimes \chi) \cong \pi(f) \otimes (\chi \circ \det)$, where $\chi : \A_f^\times/\QQ_{>0} \to \CC^\times$ denotes the adèlization of $\chi$, sending $\varpi_p$ to $\chi(p)$ for every $p \nupdownline m$. Since $\pi(f) \in \Pi(K_1(N))$, it follows that $\pi(f \otimes \chi) \in \Pi(K_1(N)_F)$.

\begin{lem}\label{lem6}
For any prime $p \nupdownline Nm$, we have
\begin{align}
\label{theta Tp} \theta_{\pi(f \otimes \chi)} (T(p)) & = a_p \chi(p) \\
\label{theta Tpp} \theta_{\pi(f \otimes \chi)} (T(p,p)) & = \chi(p)^2.
\end{align}
\end{lem}

\begin{proof}
We know that $\theta_{\pi(f)} (T(p)) = a_p$ and $\theta_{\pi(f)} (T(p,p))= 1$. The equalities (\ref{theta Tp}) and (\ref{theta Tpp}) follow formally from the fact that $\chi \circ \det$ is equal to $\chi(p)$ on the double coset $K_1(N)_F \begin{pmatrix} \varpi_p & 0 \\ 0 & 1 \end{pmatrix} K_1(N)_F$.
\end{proof}

Let $e_{f \otimes \chi} : \Omega^1(X_1(N)_F) \otimes \Qb \to \Omega_{\pi(f \otimes \chi)}^{K_1(N)_F}$ be the projection induced by (\ref{dec Omega1}). The multiplicity one theorems imply that $e_{f \otimes \chi} \in \T$.

\begin{pro}\label{divisibility}
The element $e_\chi e_F$ is divisible by $e_{f \otimes \chi}$ in $\mathcal{T}G$.
\end{pro}

\begin{proof}
Since $e_\chi$, $e_F$ and $e_{f \otimes \chi}$ are commuting projections, it suffices to prove that the image of $e_\chi e_F$ is contained in the image of $e_{f \otimes \chi}$. We know that the image of $\varphi^* : \Omega^1(E) \to \Omega^1(X_1(N))$ lies in the kernel of $T(p)-a_p \in \TT_{X_1(N)}$. Therefore the image of $\varphi_F^*$ lies in the kernel of $\nu_F(T(p))-a_p$. Using Lemma \ref{lem4}, it follows that in $\mathcal{T}G$ we have
\begin{equation}
T(p) \sigma_p e_F = a_p e_F.
\end{equation}
Applying $e_\chi$ to both sides and using the identity $e_\chi \sigma_p = \overline{\chi}(p) e_\chi$ yields
\begin{equation}
T(p) e_\chi e_F = a_p \chi(p) e_\chi e_F.
\end{equation}
The same argument shows that $T(p,p) e_\chi e_F = \chi(p)^2 e_\chi e_F$. The proposition now follows from Lemma \ref{lem6} and the multiplicity one theorems.
\end{proof}

\section{Proof of the main results}

Recall that $\varphi : X_1(N) \to E$ is a modular parametrization, and that $\varphi_F$ is the base change of $\varphi$ to $F$. We have a commutative diagram
\begin{equation}
\xymatrix{
K_2(X_1(N)_F) \otimes \QQ \ar[r] \ar[d]_{(\varphi_F)_*} & H^1(X_1(N)_F(\CC),\RR)^{-} \ar[d]^{(\varphi_F)_*}\\
K_2(E_F) \otimes \QQ \ar[r] & H^1(E_F(\CC),\RR)^{-}
}
\end{equation}
where the horizontal maps are the regulator maps on $X_1(N)_F$ and $E_F$.

The strategy of the proof is to use Beilinson's theorem on $X_1(N)_F$ and then to get back to $E_F$ using the Hecke algebra.

Let $\mathcal{P}_{E/F} = (\varphi_F)_* \mathcal{P}_{X_1(N)/F} \subset K_2(E_F) \otimes \QQ$. By \cite[1.1.2(iii)]{schappacher-scholl}, we have $\mathcal{P}_{E/F} \subset H^2_{\mathcal{M}/\ZZ}(E_F,\QQ(2))$. We want to prove that $R_{E/F} := \reg_{E/F} (\mathcal{P}_{E/F})$ is a $\QQ$-structure satisfying (\ref{det REF}). Since $\mathcal{P}_{X_1(N)/F}$ is stable by the Hecke algebra, the spaces $\mathcal{P}_{E/F}$ and $R_{E/F}$ are stable by $G$.

For any $\chi \in \widehat{G}$, let $R_\chi = e_\chi(R_{E/F} \otimes \Qb)$ and $H_\chi = e_\chi (H_{E/F} \otimes \Qb)$. We want to compare $R_\chi$ and $H_\chi$. We have

\begin{align}
\nonumber \varphi_F^* R_\chi & = e_\chi \varphi_F^* (R_{E/F} \otimes \Qb)\\
\label{phiF Rchi} & = e_\chi e_F \reg_{X_1(N)/F} (\mathcal{P}_{X_1(N)/F} \otimes \Qb).
\end{align}
Similarly, we have
\begin{equation}
\label{phiF Hchi} \varphi_F^* H_\chi = e_\chi e_F (H_{X_1(N)/F} \otimes \Qb).
\end{equation}

We will build on the following theorem of Schappacher and Scholl. Let $\lambda_\chi$ be the unique element of $(\Qb \otimes \RR)^\times$ such that for every $\psi : \Qb \hookrightarrow \CC$, we have $\psi(\lambda_\chi) = L'(f \otimes \chi^\psi,0) \in \CC^\times$.  By \cite[1.2.4 and 1.2.6]{schappacher-scholl}, we have
\begin{equation}\label{thm schappacher scholl}
e_{f \otimes \chi} \bigl(\reg_{X_1(N)/F} (\mathcal{P}_{X_1(N)/F} \otimes \Qb) \bigr) = \lambda_\chi \cdot e_{f \otimes \chi} (H_{X_1(N)/F} \otimes \Qb).
\end{equation}
By Prop. \ref{divisibility}, the equality (\ref{thm schappacher scholl}) remains true when $e_{f \otimes \chi}$ is replaced by $e_\chi e_F$, so that $\varphi_F^* R_\chi = \lambda_\chi \cdot \varphi_F^* H_\chi$ by (\ref{phiF Rchi}) and (\ref{phiF Hchi}). Since $\varphi_F^*$ is injective, we get $R_\chi = \lambda_\chi \cdot H_\chi$. Put $V = H^1(E_F(\CC),\RR)^-$ and $V_\chi = e_\chi(V \otimes \Qb)$ for any $\chi \in \widehat{G}$.

\begin{lem}\label{H1EF RGmod}
The $\RR[G]$-module $V$ is free of rank $1$.
\end{lem}

\begin{proof}
By Poincaré duality $V \cong \Hom_\QQ(\Omega^1(E_F),\RR)$, and $\Omega^1(E_F) \cong \Omega^1(E) \otimes F$ is free of rank $1$ over $\QQ[G]$ by the normal basis theorem.
\end{proof}

We will use the following lemma from linear algebra. Recall that if $B$ is an $A$-algebra and $N$ is a $B$-module, an \emph{$A$-structure of $N$} is an $A$-submodule $M \subset N$ such that $M \otimes_A B \xrightarrow{\cong} N$.

\begin{lem}\label{QG structure}
Let $M$ be a $\QQ[G]$-submodule of $V$. The following conditions are equivalent :
\begin{enumerate}
\item[(i)] $M$ is a $\QQ$-structure of the real vector space $V$.
\item[(ii)] For any $\chi \in \widehat{G}$, the space $M_\chi := e_\chi (M \otimes \Qb)$ is a $\Qb$-structure of the $\Qb \otimes \RR$-module $V_\chi$.
\end{enumerate}
Moreover, if these conditions hold, then $M$ is free of rank $1$ over $\QQ[G]$.
\end{lem}

\begin{proof}
The implication $(i) \Rightarrow (ii)$ follows from the isomorphisms $M_\chi \otimes_{\Qb} (\Qb \otimes \RR) \cong e_\chi (M \otimes \Qb \otimes \RR) \cong V_\chi$. Let us assume $(ii)$. By Lemma \ref{H1EF RGmod}, the $\Qb \otimes \RR$-module $V_\chi$ is free of rank $1$, so that $\dim_{\Qb} M_\chi = 1$. Since $M \otimes \Qb \cong \bigoplus_{\chi \in \widehat{G}} M_\chi$, we get $\dim_\QQ M = d$. Moreover $M \otimes \Qb \otimes \RR$ generates $V \otimes \Qb$ over $\Qb \otimes \RR$, so that any $\QQ$-basis of $M$ is actually free over $\RR$.

Finally, if $(i)$ holds, then $M$ is isomorphic to the regular representation of $G$ by Lemma \ref{H1EF RGmod}, so that $M$ is free of rank $1$ over $\QQ[G]$.
\end{proof}

Using Lemma \ref{QG structure} with the $\QQ$-structure $H_{E/F}$, we see that $H_\chi$ is a $\Qb$-structure of $V_\chi$. By Lemma \ref{eta chi}, the $1$-form $\eta_\chi$ is a $\Qb$-basis of $H_\chi$.

\begin{proof}[Proof of Theorem \ref{regEF thm}]
Since $R_\chi = \lambda_\chi \cdot H_\chi$ is a $\Qb$-structure of $V_\chi$, Lemma \ref{QG structure} implies that $R_{E/F}$ is a $\QQ$-structure of $V$. Moreover, the determinant of $R_{E/F} \otimes \Qb$ with respect to $H_{E/F} \otimes \Qb$ is represented by $\delta := \prod_{\chi \in \widehat{G}} \lambda_\chi \in (\Qb \otimes \RR)^\times$. Note that $\sigma(\lambda_\chi) = \lambda_{\chi^\sigma}$ for any $\sigma \in \Gal(\Qb/\QQ)$, so that $\delta$ lies in the image of $\RR^\times$ in $(\Qb \otimes \RR)^\times$. Using the natural evaluation map $(\Qb \otimes \RR)^\times \xrightarrow{\iota} \CC^\times$, we get in fact $\delta = \prod_{\chi \in \widehat{G}} L'(f \otimes \chi,0)$. Since the natural map $\RR^\times / \QQ^\times \to (\Qb \otimes \RR)^\times / \Qb^\times$ is injective, we conclude that $\det(R_{E/F}) = L^{(d)}(E/F,0) \cdot \det(H_{E/F})$ by (\ref{LEF0}).
\end{proof}

\begin{proof}[Proof of Theorem \ref{LEchi0 thm}]
We know from Theorem \ref{regEF thm} that $R_{E/F}$ is a $\QQ$-structure of $V$. Since $R_{E/F}$ is stable by $G$, it is free of rank $1$ over $\QQ[G]$ by Lemma \ref{QG structure}. Let $\gamma \in \mathcal{P}_{E/F}$ such that $R_{E/F} = \QQ[G] \cdot \reg_{E/F}(\gamma)$. Replacing $\gamma$ by a suitable integer multiple, we may assume that $\gamma$ has a representative $\widetilde{\gamma} \in F(E)^\times \otimes F(E)^\times$. Let $\ell = \beta(\widetilde{\gamma})$. For any $\chi \in \widehat{G}$, we have $R_\chi = \mu_\chi(\ell) H_\chi$ by Prop. \ref{echi reg}, where $\mu_\chi(\ell)$ is given by (\ref{eq echi reg}). It follows that $\mu_\chi(\ell)/\lambda_\chi \in \Qb^\times$. Since $\lambda_\chi$ and $\mu_\chi(\ell)$ belong to $\QQ(\chi) \otimes \RR$, we have in fact $\mu_\chi(\ell)/\lambda_\chi \in \QQ(\chi)^\times$. Moreover, the definitions of $\lambda_\chi$ and $\mu_\chi(\ell)$ show that
\begin{equation}
\tau(\lambda_\chi) = \lambda_{\chi^\tau} \quad \textrm{and} \quad \tau(\mu_\chi(\ell)) = \mu_{\chi^\tau}(\ell) \qquad (\tau \in \Gal(\QQ(\chi)/\QQ)).
\end{equation}

\begin{lem}\label{lemme alg}
Let $(a_\chi)_{\chi \in \widehat{G}}$ be a family of algebraic numbers, with $a_\chi \in \QQ(\chi)^\times$, such that $\tau(a_\chi) = a_{\chi^\tau}$ for any $\chi$ and any $\tau \in \Gal(\QQ(\chi)/\QQ)$. Then there exists a unique $a \in \QQ[G]^\times$ such that for every $\chi \in \widehat{G}$, we have $\chi(a) = a_\chi$.
\end{lem}

\begin{proof}
The canonical morphism of $\QQ$-algebras $\Psi : \QQ[G] \to \prod_{\chi \in \widehat{G}} \QQ(\chi)$ is injective and its image is contained in the subalgebra $W$ of families $(b_\chi)_{\chi}$ satisfying $\tau(b_\chi) = b_{\chi^\tau}$ for any $\chi$ and $\tau$. Writing $\widehat{G}$ as a disjoint union of Galois orbits, we have $\dim_\QQ W = \# \widehat{G} = d$, so that $\Psi$ is an isomorphism.
\end{proof}

Using Lemma \ref{lemme alg} with $a_\chi := \mu_\chi(\ell)/\lambda_\chi$, we get $a \in \QQ[G]^\times$ such that $\mu_\chi(\ell) = \chi(a) \lambda_\chi$ for any $\chi$. Since $\mu_\chi(a\ell) = \overline{\chi}(a) \mu_\chi(\ell)$, replacing $\ell$ with a suitable integer multiple of $a\ell$ results in $\mu_\chi(\ell) \sim_{\QQ^\times} \lambda_\chi$ for any $\chi$. Evaluating everything in $\CC$ yields (\ref{LEchi0 formula}).
\end{proof}


\begin{proof}[Proof of Corollary]
Let us first recall the Dedekind-Frobenius formula for group determinants. If $a : G \to \CC$ is an arbitrary function, let $A$ be the matrix $(a(gh^{-1}))_{g,h \in G}$. Then
\begin{equation}\label{frobenius determinant}
\det(A) = \prod_{\chi \in \widehat{G}} \sum_{g \in G} \chi(g) a(g).
\end{equation}

Let $\ell \in \ZZ[E(\Qb)]^{G_F}$ be a divisor satisfying the identities (\ref{LEchi0 formula}) of Theorem \ref{LEchi0 thm}. Assume first $F$ is real. Put $\ell_i := \ell^{\sigma_i^{-1}}$ for $1 \leq i \leq d$. Using (\ref{frobenius determinant}) with $a(\sigma)=D_E(\ell^\sigma)$ yields
\begin{equation}
\det \bigl(D_E(\ell_i^{\sigma_j})\bigr)_{1 \leq i,j \leq d} \sim_{\QQ^\times} \prod_{\chi \in \widehat{G}} \pi L'(E \otimes \chi,0) \sim_{\QQ^\times} \pi^{-d} L(E/F,2)
\end{equation}
where the last relation follows from (\ref{LEF0}) and Prop. \ref{LEF20}.

Assume now $F$ is complex. Put $\ell_i := \ell^{\sigma_i^{-1}}$ for $1 \leq i \leq d/2$. Let us use (\ref{frobenius determinant}) with the function $a(\sigma)=D_E(\ell^\sigma)+J_E(\ell^\sigma)$. Indexing the lines and columns of $A$ by $\sigma_1,\overline{\sigma_1},\ldots,\sigma_{d/2},\overline{\sigma_{d/2}}$, we see that $A$ consists of $\frac{d}{2}$ blocks of the forms $\begin{pmatrix} x+y & x-y \\ x-y & x+y \end{pmatrix}$, where $x=D_E(\ell^{\sigma_j \sigma_i^{-1}})$ and $y=J_E(\ell^{\sigma_j \sigma_i^{-1}})$. Elementary operations on the lines and columns of $A$ thus gives
\begin{equation}
\det A = 2^d \det \bigl(D_E(\ell_i^{\sigma_j})\bigr)_{1 \leq i,j \leq d/2} \cdot \det \bigl(J_E(\ell_i^{\sigma_j})\bigr)_{1 \leq i,j \leq d/2}.
\end{equation}
On the other hand, we have
\begin{equation}
\sum_{\sigma \in G} \chi(\sigma) a(\sigma) = \begin{cases} \sum_{\sigma \in G} \chi(\sigma) D_E(\ell^\sigma) & \textrm{if } \chi \textrm{ is even},\\
\sum_{\sigma \in G} \chi(\sigma) J_E(\ell^\sigma) & \textrm{if } \chi \textrm{ is odd},
\end{cases}
\end{equation}
so that we conclude as in the first case.
\end{proof}

\section*{Further remarks and open questions}

During the course of proving Theorem \ref{regEF thm}, we crucially needed the fact that $F/\QQ$ is abelian, in order for the curve $X_1(N)_F$ to be itself a modular curve. Another key part of the argument was to realize the $\chi$-part of the cohomology of $E$ as a subspace of a suitable Hecke eigenspace. Note that this subspace can be strict, because of the existence of old forms or because it can happen that $f \otimes \chi = f$ (for example when $E$ has complex multiplication).

Note that if the extension $F/\QQ$ isn't abelian, the curve $X_1(N)_F$ might not be covered by a modular curve. In this case, we don't know how to prove a single example of Zagier's conjecture for $E/F$. We also have no example in the case of an elliptic curve over a number field which isn't the base extension of an elliptic curve over $\QQ$.

It would be interesting to investigate the rational factors appearing in (\ref{LEchi0 thm}). These factors might be linked with the Bloch-Kato conjectures for $L(E/F,2)$. However, even for elliptic curves over $\QQ$, we don't know of a precise conjecture predicting the rational factor appearing in Zagier's conjecture.

Although the divisor $\ell$ produced by Theorem \ref{LEchi0 thm} is inexplicit in general, it would be interesting to try to bound the number field generated by the support of $\ell$, as well as the heights of these points.

Finally, Theorem \ref{LEchi0 thm} suggests to formulate an equivariant version of Zagier's conjecture for base extensions of elliptic curves, along the lines of the equivariant Tamagawa number conjecture of Burns and Flach \cite[Part 2, Conjecture 3]{flach}. For example, in the case $F$ is real, Theorem \ref{LEchi0 thm} gives a link between the equivariant $L$-function of $E_F$ evaluated at $2$, which is an element of $\RR[G]$, and the vector-valued elliptic dilogarithm $\Dvec(\ell) := \sum_{\sigma \in G} D_E(\ell^\sigma) [\sigma]$.

\bibliographystyle{plain}
\bibliography{zagierLEF}

\end{document}